\documentclass[a4paper,11pt,oneside,openright,reqno]{amsart}
\usepackage[english]{babel}
\usepackage[hidelinks]{hyperref}
\usepackage{amsthm}
\usepackage{color}
\usepackage{mathrsfs}
\usepackage{amsmath}
\usepackage{mathtools}
\usepackage{amsfonts}
\usepackage{amssymb}
\usepackage{bm}
\usepackage{physics}
\usepackage{enumitem}
\usepackage{cleveref}
\usepackage{esint}
\usepackage{nicefrac}
\usepackage{yfonts}
\usepackage{dirtytalk}

\numberwithin{equation}{section}

\newtheorem{thm}{Theorem}
\newtheorem*{thm*}{Theorem}
\newtheorem*{qest*}{Question}

\newtheorem{lem}[thm]{Lemma}
\newtheorem{prop}[thm]{Proposition}

\theoremstyle{definition}

\theoremstyle{remark}

\newcommand{\defeq}{\mathrel{\mathop:}=}

\def\XXint#1#2#3{{\setbox0=\hbox{$#1{#2#3}{\int}$}
		\vcenter{\hbox{$#2#3$}}\kern-.5\wd0}}

\newcommand{\RR}{\mathbb{R}}
\newcommand{\CC}{\mathbb{C}}
\newcommand{\NN}{\mathbb{N}}
\newcommand{\ZZ}{\mathbb{Z}}

\newcommand{\Ric}{{\mathrm{Ric}}}
\let\epsilon\varepsilon

\author[C.\ Brena]{Camillo Brena}
\address{School of Mathematics, Institute for Advanced Study, 1 Einstein Dr., Princeton NJ 08540, USA}
\email{cbrena@ias.edu}

\title[]{Instability of the fundamental group \\for non-collapsed Ricci-limits}

\begin{document}\maketitle
\begin{abstract}
We construct two sequences of closed $4$-dimensional manifolds with non-negative Ricci curvature,  diameter bounded from above by $1$, and volume bounded from below by $v>0$, with different fundamental groups but with the same Gromov--Hausdorff limit. This provides a negative answer to the question posed in [J.\ Pan. Ricci Curvature and Fundamental Groups of Effective Regular Sets. \textit{Journal of Mathematical Study}, 58(1):3--21, 2025]. 
%
%
\end{abstract}
	\section*{Introduction}
Starting from the works of Cheeger--Colding \cite{Cheeger-Colding97I,Cheeger-Colding97II,Cheeger-Colding97III}, significant effort has been devoted to studying spaces that arise as Gromov--Hausdorff limits $M_i\stackrel{GH}{\rightarrow} X$, where $M_i$ are $n$-dimensional Riemannian manifolds satisfying 
\begin{equation}\label{assumptions}
    \Ric_{M_i}\ge -(n-1),\quad\mathrm{diam}(M_i)\le D,\quad\mathrm{vol}(M_i)\ge v>0.
\end{equation}
Among the various natural questions arising from this theory, there is the one about the stability of fundamental groups. Notice  first that by \cite{peralex}, \cite{GPW}, if one replaces the  Ricci curvature bound $\Ric_{M_i}\ge -(n-1)$ in \eqref{assumptions} with the sectional curvature bound $\mathrm{sec}_{M_i}\ge -1$, then $M_i$ and $X$ are homeomorphic for large $i$, in particular, they have isomorphic fundamental groups.
However, this is not true in the case that we are considering, i.e., lower bounds on the Ricci curvature. 
For example, the stability of the fundamental groups fails in the example constructed  in \cite{Otsu91}, which consists of a sequence of metrics $(g_i)$ on $S^3\times \RR\mathbb{P}^2$ satisfying \eqref{assumptions} and such that $(S^3\times\RR\mathbb{P}^2,g_i)$ converges to the \emph{singular} space given by the spherical suspension of $S^2\times \RR\mathbb{P}^2$, which is simply connected. On the other hand, there are some positive results in this direction. For instance, for a sequence $M_i\stackrel{GH}{\rightarrow} X$ as in \eqref{assumptions}, there exists, for $i$ large, a surjective group homomorphism $\pi_1(M_i)\rightarrow\pi_1(X)$, see \cite{MR4493619}. Moreover, convergence of  first Betti numbers holds, by the results recalled in \cite{MR4595312}, building upon \cite{zamora2022bettinumbercollapse,MR1837249,MR4810225}.

In this note, we  consider the following  two questions,  posed by J.\ Pan  in \cite{pan24}.
\begin{qest*}
Given $n,D,v>0$, is there a positive constant $\epsilon=\epsilon(n,D,v)>0$ such that     if two closed $n$-dimensional   manifolds $M_1$ and $M_2$ satisfy \eqref{assumptions} and $d_{GH}(M_1,M_2)<\epsilon$, then the fundamental groups of $M_1$ and $M_2$ are isomorphic?
\end{qest*}
\begin{qest*}
For a convergent sequence of closed $n$-dimensional manifolds $M_i\stackrel{GH}{\rightarrow} X$ satisfying \eqref{assumptions}, is it possible to determine the fundamental group of $M_i$, for $i$ large enough, solely from $X$?
\end{qest*}
As noted in \cite{pan24}, the two questions above clearly have the same answer.
 The aim of this note is to justify a negative answer to these questions with the following example.
 
    \begin{thm*}
		There exist two sequences of closed $4$-dimensional  manifolds, $(M_i)$ and $(N_i)$, with
		\begin{equation*}
			d_{GH}(M_i,N_i)\rightarrow 0,\qquad\text{as $i\rightarrow\infty$},
		\end{equation*} 
		and such that, for a positive constant $v>0$, for every $i$, the  following hold.
		\begin{itemize}
			\item $M_i$ and $N_i$ have non-negative Ricci curvature.
			\item The diameters of $M_i$ and $N_i$  are bounded from above by $1$.
			\item The volumes of $M_i$ and $N_i$ are bounded from below by $v$.
			\item $M_i$ has fundamental group $\ZZ/2\ZZ$, whereas $N_i$ is simply connected.
		\end{itemize}
	\end{thm*}

    \subsection*{Overview of the example}
	Consider $S^3=\{(z_1,z_2):|z_1|^2+|z_2|^2=1\}\subseteq\CC^2$, endowed with the standard metric $ds_3^2$. We let $\mu_k$ be the group of isometries of $\CC^2$ induced by the diagonal action of the $k$-th roots of unity, which clearly acts freely on $S^3$. Notice that $\mu_2$ corresponds to  the antipodal action. 
    Our strategy is to build two sequences $(M_i)$ and $(N_i)$ with different fundamental groups but the same Gromov--Hausdorff limit.
    
\begin{itemize}
        \item[(a)] First, we will construct a  smooth simply connected manifold $M'$ with non-negative Ricci curvature, asymptotic  at infinity to the cone over $S^3/\mu_2=\RR\mathbb{P}^3$. Moreover, $M'$ admits a free $\ZZ/2\ZZ$ action  and the quotient of $M'$ by this action, $M$, has fundamental group $\ZZ/2\ZZ$ and   is asymptotic  at infinity to the cone over $S^3/\mu_4$.
        
        We will obtain the space $M'$ starting from the Eguchi--Hanson space, which is the cone  over $S^3/\mu_2$ endowed with a smooth metric, which resolves the apparent singularity at the origin by identifying to a point the Hopf fibers and, outside a compact set, is the standard metric. Topologically, the Eguchi--Hanson space is the cotangent bundle of $S^2$, which is simply connected. Moreover, the Eguchi--Hanson space  admits a free $\ZZ/\ZZ_2$ action $\iota$, defined by
        \begin{equation}\notag
            [(z_1,z_2)]_{\mu_2}\stackrel{\iota}{\mapsto}[(-\bar z_2,\bar z_1)]_{\mu_2};
        \end{equation}
        notice that $\iota$ remains free even under the identification of the Hopf fibers.
        The quotient of the Eguchi--Hanson space by $\iota$ is isometric to  the cone over $(S^3/\mu_2)/\{1,\iota\}$ outside a compact set. 

        We will consider a transformation of $\CC^2=\RR^4$ (which does not retain neither the complex structure nor the Hopf fibration) which maps, for any $r$, isometrically, $(rS^3/\mu_2)/\{1,\iota\}$ to  $rS^3/\mu_4$. We  thus obtain $M'$ as the image of the Eguchi--Hanson space through this isometry.
        \item[(b)] Second, we will consider a smooth simply connected manifold $N$ with non-negative Ricci curvature, asymptotic at infinity  to the cone over $S^3/\mu_4$. 
        
        The space $N$ will be the desingularization of the  orbifold $\CC^2/\mu_4$. 
        There are various ways to desingularize  the  orbifold $\CC^2/\mu_4$ with a smooth metric of non-negative Ricci curvature. For example, \cite{Kronheimer89} exhibits such a metric on the $4$-dimensional manifold which underlies the  complex space $X$, where $X\rightarrow\CC^2/\mu_4$ is the minimal resolution of the quotient singularity. 
    \end{itemize}
    By considering $M_i$ and $N_i$ as rescaled copies of $M$ and $N$, respectively, for a blow-down sequence $R_i$, one obtains two sequences of $4$-dimensional manifolds with non-negative Ricci curvature converging to the cone over $S^3/\mu_4$, but with different fundamental groups. This settles our main result in the non-compact case. To obtain the case with uniformly bounded diameter, it is enough to change the metrics in a conformal way and perform gluing (cf.\ \cite{anderson1990}) to obtain the spherical suspension of  $S^3/\mu_4$ as common limit.

    To keep the construction of the example as explicit as possible, we exhibit directly  the smooth metric for the  resolution of the singularity in $\CC^2/\mu_4$, as a doubly warped product. The drawback with this construction is that we obtain, as asymptotic cone, the cone over $cS^3/\mu_4$, where $c$ is a small constant. As a result, the sequences $(M_i)$ and $(N_i)$ will converge to the spherical suspension of  $cS^3/\mu_4$.
    \section*{Acknowledgments}
    I want to thank Elia Bruè for suggesting the strategy to  construct the example. I am also grateful  to Francesco Malizia, Riccardo Ontani  and Alessandro Pigati for discussions about the topic of this paper.
    
        This material is based upon work supported by the National Science Foundation under Grant No.\ DMS-1926686.

	\section*{The Eguchi--Hanson space}
    We will recall below the classical  construction of the Eguchi--Hanson space, which will be needed to obtain the spaces  $(M_i)$. More precisely, we will perturb the metric and perform a gluing procedure.
	\begin{thm}\label{thm1}
		There exists a constant $c'>0$ and, for every $c\in (0,c')$, there exists a   Riemannian metric $ds_{M'}^2=ds_{M',c}^2$ on $(0,\infty)\times( S^3/\mu_2)$, such that the following hold.
		\begin{itemize}
			\item The metric extends to a smooth metric on the completion of the space, which is a smooth simply connected $4$-manifold $M'$.
			\item The metric $ds_{M'}^2$ has non-negative Ricci curvature.
			\item $M'$ has  a free involutive isometry, which, for $S^3\subseteq \CC^2$,  reads as
			\begin{equation*}
				(r,(z_1,z_2))\mapsto (r,(i {z_1},i{z_2})).
			\end{equation*} 
			\item It holds
			\begin{equation*}
				ds_{M'}^2=dr^2+c^2(r+1)^2ds_3^2\qquad\text{for }r>1.
			\end{equation*}
		\end{itemize}
	\end{thm}
	\begin{thm}\label{thm1a}
	There exists a constant $c'>0$ and, for every $c\in (0,c')$ and $d\in (0,1/2)$, there exists a   Riemannian metric $ds_{M'}^2=ds_{M',c,d}^2$ on $(0,\pi)\times (S^3/\mu_2)$, such that the following hold.
	\begin{itemize}
		\item The metric extends to a smooth metric on the completion of the space, which is a smooth simply connected $4$-manifold $M'$.
		\item The metric $ds_{M'}^2$ has non-negative Ricci curvature.
		\item $M'$ has  a free involutive isometry, which, for $S^3\subseteq \CC^2$,  reads as \begin{equation*}
			(r,(z_1,z_2))\mapsto (r,(i{z_1},i{z_2})).
		\end{equation*} 
		\item It holds
		\begin{equation*}
			ds_{M'}^2=dr^2+c^2\big(\sin(r)-\sin\big(9d/20\big)+ d\big)^2ds_3^2\qquad\text{for }r\in(d,\pi-d).
		\end{equation*}	
	\end{itemize}
\end{thm}
The rest of this section is devoted to the proof of the two theorems above.
	\subsection*{The Eguchi--Hanson construction}
	For $a\in(0,1)$, we use the  variables
	\begin{equation*}
		r \in (a,\infty),\quad
		\theta\in (0,\pi),\quad	\phi\in (0,2\pi),\quad	\psi\in (0,4\pi).
	\end{equation*}
	We also recall the following definitions
	 \begin{align*}
		\sigma_x  &= 
		 \frac{1}{2} \big(\sin (\psi) d\theta - \sin (\theta) \cos (\psi) d\phi\big),\\
		\sigma_y &= 
		  \frac{1}{2}\big(- \cos (\psi) d\theta-\sin(\theta)\sin(\psi) d\phi\big),\\
		\sigma_z &=  
		 \frac{1}{2} \big(d\psi + \cos (\theta) d\phi\big).
	\end{align*}
	For every $r$, we isometrically embed $(S^3,r^2(\sigma_x^2+\sigma_y^2+\sigma_z^2))$  in  $\mathbb{C}^2$ by
	\begin{equation*}
		z_1=r\cos\big(\tfrac\theta 2\big)e^{\tfrac{i}{2}(\psi+\phi)},\qquad
		z_2=r\sin\big(\tfrac\theta 2\big)e^{\tfrac{i}{2}(\psi-\phi)}.
	\end{equation*}

	Eguchi and  Hanson considered in \cite{EguchiHanson} the gravitational instanton   given by \begin{equation*}
		(a,\infty)\times( S^3/\{\psi\sim \psi+2\pi\}),
	\end{equation*}  with  the  metric
	\begin{equation}\label{verfdsfcas}
		ds^2=\frac{1}{1-({a/r})^4}dr^2+r^2 (\sigma_x^2+\sigma_y^2)+r^2\big(1-({a}/{r})^4\big)\sigma_z^2.
	\end{equation}
	The quotient by  the relation $\psi\sim \psi+2\pi$ corresponds to the quotient by   $\mu_2$  on $\CC^2$.  
	As an effect of the degeneration of the metric as $r\downarrow a$, we have that, at $r=a$, all the points $\{(r,\theta,\phi,\psi)\}_{\psi\in [0,2\pi]}$ have to be identified, thus corresponding to a two sphere $S^2$ with coordinates $(\theta,\phi)$. 
	In complex coordinates, the points on the two sphere correspond to the orbits of the diagonal action of $SU(1)$ on $\CC^2$.
	We can  thus see the Eguchi--Hanson space as the cotangent bundle of this two sphere, where the coordinates $(r,\psi)$ correspond to coordinates on the cotangent fibers $\RR^2$. In particular, the space is simply connected (more precisely, there is an obvious deformation retraction from the space to a two sphere).
	
	We finally remark that setting 
	\begin{equation*}
			u=r\big(1-(a/r)^4\big)^{1/2}\in (0,\infty),
	\end{equation*}
	the metric becomes
	\begin{equation*}
		ds^2=\frac{1}{\big(1+({a}/{r})^4\big)^2}du^2+r^2 (\sigma_x^2+\sigma_y^2)+u^2\sigma_z^2.
	\end{equation*}
	\subsection*{A free $\ZZ/2\ZZ$ isometry}
	We consider the map $\iota$, involution on $S^3/\{\psi\sim \psi+2\pi\}=S^3/\mu_2$  defined as
	\begin{equation*}
		\theta\rightarrow\pi-\theta,\quad
		\phi\rightarrow\pi+\phi,\quad
		\psi\rightarrow \pi-\psi,
	\end{equation*}
	where $\phi$ and $\psi$   have to be understood modulo $2\pi$.  
    Notice that this map is \emph{not} well defined on $S^3$ without the identification $\psi\sim\psi+2\pi$. This because the action $\phi\mapsto\pi+\phi$ forces us to consider the $\phi$ coordinate modulo $2\pi$, so that, in order to recover the other coordinate $\psi$ uniquely, it is necessary to take also the second quotient modulo $2\pi$ and not $4\pi$.
    However, $\iota$ is induced by  well-defined map on $\CC^2$, that we still call $\iota$, which reads as
	\begin{equation*}
		(z_1,z_2)\mapsto (-\bar z_2, \bar z_1).
	\end{equation*}
	Notice that $\iota$ has  order $4$ on $\CC^2$, $\iota^2$  is $-1$ acting diagonally, and $\iota$ is well-defined and of order $2$ on $\CC^2/\mu_2$. 
	On the Eguchi--Hansen space, $\iota$ is an involutive  isometry (this is easier to prove in angular coordinates)   without fixed points. 
	We define $\nu_4$ as the group of isometries of $\RR^4=\CC^2$ generated by $\iota$. 
    
    We use coordinates $(x_1,x_2,x_3,x_4)$ on $\RR^4=\CC^2$. Consider the map 
    \begin{equation}\notag
        \Psi:S^3\rightarrow S^3\qquad(x_1,x_2,x_3,x_4)\mapsto(x_1,x_3,x_2,-x_4).
    \end{equation} 
    Then $\Psi\circ\iota\circ\Psi^{-1}$ reads as the multiplication by $i$, diagonally, i.e., the action of $\mu_4$. Moreover, the action of $\mu_2$ commutes with  $\Psi$.

\subsection*{Modification of  the metric}

	We now modify the metric following \cite[Proof of Proposition 3.1]{anderson1990}.  We set $t=0$ on $u=0$ and
	\begin{equation*}
		t=\int \frac{du}{1+(a/r)^4}.
	\end{equation*}
	 We set $d s_h^2=1/h(t)^2 ds^2$, where $h$ is a suitable convex function such that
	\begin{equation*}
		h(s)=
		\begin{cases}
			1\qquad&\text{for $s\in (0,1/4)$},\\
			1+s^2\qquad&\text{in a neighborhood of $1$}.
		\end{cases}
	\end{equation*}
	The function $h$ is chosen as in \cite{anderson1990}, so that $ds_h^2$ has non-negative Ricci curvature. 
	
	We first remark that, as $a\downarrow 0$, the metric $ds_h^2$ smoothly converges, on a neighborhood of $\{u=1\}$, to the metric
	\begin{equation}\label{fdscsac}
		\frac{1}{(1+u^2)^2}du^2+\frac{u^2}{(1+u^2)^2}d s_3^2,
	\end{equation}
	 which   has Ricci curvature bounded from below by $12$ (this is apparent after recalling \eqref{verfdsfcas} and performing the change of variables described below, which establishes an isometry with an open subset of a sphere). 
	 Then, if $a$ is small enough, we can perturb the metric with a compactly supported deformation near $\{u=1\}$, to obtain a metric $d\tilde s^2$, that equals \eqref{fdscsac} in a neighborhood of $\{u=1\}$, and still has non-negative Ricci curvature.  
	 
	We set now $u=\tan(\rho/2 )$, so that $d\tilde s$,  for some $b'\in (0,1)$, 
	reads as
	\begin{equation*}
		d\tilde s^2=\tfrac14d\rho^2+\tfrac14\sin^2(\rho) ds_3^2\qquad\text{for $\rho\in (\pi/2-b',\pi/2+b')$}.
	\end{equation*}
	Let now $b\in (0,b')$ and consider the function 
	\begin{equation*}
		\hat\varphi(\rho)\defeq 
		\begin{cases}
			 \sin(\rho)\qquad &\rho\in (0,\pi/2-b),\\
			\tfrac12\cos\big(\pi/2-b\big)(\rho-\pi/2+b)+\sin\big(\pi/2-b\big)\qquad&\rho\in (\pi/2-b,\infty).
		\end{cases}
	\end{equation*}
   Using \cite[Lemma 1.5]{Otsu91} we can perturb $\hat\varphi$ in a small neighborhood of $\pi/2-b$ to obtain a smooth concave and $1$-Lipschitz function  $\varphi$. Hence,  we can  set
	\begin{equation*}
		d\hat s^2=
		\begin{cases}
			d\tilde s^2\qquad&\rho\in (0,\pi/2-b'),\\
			\tfrac14 d\rho^2+\tfrac14\varphi(\rho)^2ds_3^2\qquad &\rho\in (\pi/2-b',\infty),
		\end{cases}
	\end{equation*}
	which, by Lemma \ref{Ricwarp1} below, still has  non-negative Ricci curvature.

    The following result is  well-known.
    \begin{lem}\label{Ricwarp1}
		On $(0,\infty)\times S^k$, consider the Riemannian metric
		\begin{equation*}\notag
			g=dr^2+\varphi(r)^2 ds_k^2.
		\end{equation*}
		Then, if $X_0=\nabla r$ and $X_1$ is tangent to $S^k$, it holds
		\begin{align*}
			\Ric(X_0)=& -k\frac{\varphi''}{\varphi}X_0,\\
			\Ric(X_1)=&\bigg(-\frac{\varphi''}{\varphi}+(k-1)\frac{1-(\varphi')^2}{\varphi^2}\bigg)X_1.
		\end{align*}
	\end{lem}

\subsection*{Proofs of Theorems \ref{thm1} and  \ref{thm1a}}
	 Theorem \ref{thm1} follows by using the metric $d\hat s^2$ constructed above, with $c=\cos(\pi/2-b)/2$, where $b\in (0,b')$ and $b'$ is fixed as above. To be more precise, the variable $\rho$ and the metric have to be  rescaled to obtain the statement of Theorem \ref{thm1}. This proves the theorem up to the fact that the free involutive isometry is given by $(r,(z_1,z_2))\mapsto (r,(-\bar z_2,\bar z_1))$. To obtain the correct form of the isometry, it is enough to use the map $\Psi$ on $\RR^4=\CC^2$  as described above (i.e., consider the image of the space through $\Psi$).	\qed

\noindent
To prove Theorem \ref{thm1a}, we rescale the metric $ds_{M',c}^2$ and the variable $r$ given by 	Theorem \ref{thm1}, thus obtaining $d\tilde s_{M',c,d}^2$ such that 
	\begin{equation*}
		d\tilde s_{M',c,d}^2=dr^2+(c(r+d/10))^2 ds_3^2\qquad\text{for $r>d/2$}.
	\end{equation*}
	Next, we consider the function 
	\begin{equation*}
		\hat\psi(\rho)\defeq 
		\begin{cases}
		c(r+d/10)\qquad &r\in (d/2,9d/10),\\
		c/2\sin(r)-c/2\sin(9d/10)+ cd\qquad&r\in (9d/10,\pi-9d/10),\\
		c(\pi- r+d/10)\qquad &r\in (\pi-9d/10,\pi-d/2).\\
		\end{cases}
	\end{equation*} 
    By \cite[Lemma 1.5]{Otsu91}, we can perturb $\hat\psi$ in a small neighborhood of $\{9d/10,\pi-9d/10\}$ to obtain a smooth concave  and $1$-Lipschitz function $\psi:(d/2,\pi-d/2)$.
	We set
	\begin{equation*}
	d s_{M',c,d}^2=
	\begin{cases}
		d\tilde s_{M',c,d}^2\qquad&r\in (0,d/2) ,\\
		dr^2+\psi(r)^2ds_3^2\qquad&r\in (d/2,\pi-d/2), \\
		d\tilde s_{M',c,d}^2(\pi-r,\cdot)\qquad&r\in (\pi-d/2,\pi) ,\\
	\end{cases}
\end{equation*}
which has still non-negative Ricci curvature thanks to Lemma \ref{Ricwarp1}. The resulting space, which corresponds to a gluing,  remains simply connected, by the Seifert--Van Kampen Theorem, and the free involutive isometry is naturally extended.	\qed

	\section*{Desingularization of $\CC^2/\mu_4$}
    In Theorem \ref{thm2a}  we describe the spaces $(N_i)$, which are obtained from the spaces described in Theorem \ref{thm2}.
\begin{thm}\label{thm2}
	There exists a constant $c>0$ and, for every $c\in (0,c')$, there exists a   Riemannian metric $ds_N^2=ds_{N,c}^2$ on $(0,\infty)\times (S^3/\mu_4)$, such that the following hold.
	\begin{itemize}
		\item The metric extends to a smooth metric on the completion of the space, which is a smooth simply connected $4$-manifold $N$.
		\item The metric $ds_N^2$ has non-negative Ricci curvature.
		\item  It holds 
		\begin{equation*}
		ds_N^2=	dr^2+c^2(r+1)^2ds_3^2\qquad\text{for }r>1.
		\end{equation*}
	\end{itemize}
\end{thm}

\begin{thm}\label{thm2a}
	There exists a constant $c>0$ and, for every $c\in (0,c')$ and $d\in (0,1/2)$, there exists a   Riemannian metric $ds_N^2=ds_{N,c,d}^2$ on $(0,\pi)\times (S^3/\mu_4)$, such that the following hold.
	\begin{itemize}
		\item The metric extends to a smooth metric on the completion of the space, which is a smooth simply connected $4$-manifold $N$.
		\item The metric $ds_N^2$ has non-negative Ricci curvature.
		\item It holds
		\begin{equation*}
			ds_N^2=dr^2+c^2\big(\sin(r)-\sin\big(9d/20\big)+d\big)^2ds_3^2\qquad\text{for }r\in(d,\pi-d).
		\end{equation*}	
	\end{itemize}
\end{thm}
The rest of this section is devoted to the proof of the two theorems above.

\subsection*{The Berger sphere}
Now we use the  variables  
\begin{equation*}
	r\in (0,\infty),\quad\xi\in (0,\pi/2),\quad\alpha\in (0,2\pi),\quad\beta\in (0,2\pi).
\end{equation*}
Notice that for every $r$,    $S^3$ is  embedded in $\CC^2$ by
\begin{equation*}
	z_1=r\sin(\xi)e^{i\alpha}\qquad z_2=r\cos(\xi)e^{i(\alpha+\beta)}.
\end{equation*}
The Hopf fibration  $\pi_{Hopf}$ is defined by
\begin{equation*}
	(\xi,\alpha,\beta)\mapsto \bigg(\frac{\sin(2\xi)}{2}e^{-i\beta}, \frac{\cos(2\xi)}{2}\bigg)\in\CC^2.
\end{equation*}
In complex coordinates, this reads as 
\begin{equation*}
	(z_1,z_2)\mapsto \Big(z_1\bar z_2,-\tfrac{1}{2}\big(|z_1|^2-|z_2|^2\big)\Big).
\end{equation*}
In the coordinates $(\xi,\beta)$ induced by the Hopf fibration,  the standard  metric on $S^2$ reads  as  
\begin{equation*}
	ds_2^2=d\xi^2+\bigg(\frac{\sin(2\xi)}{2}\bigg)^2d\beta^2.
\end{equation*}
For $\rho$ and $\varphi$ to be fixed later, we write the metric on $(0,\infty)\times S^3$ as
\begin{equation*}
	ds^2_{\rho,\phi}=dr^2+\rho(r)^2d\alpha^2+\phi(r)^2\pi_{Hopf}^* ds_2^2;
\end{equation*}
notice that $d\alpha^2+\pi_{Hopf}^* ds_2^2=ds_3^2$, in the sense that this choice  makes the embedding $S^3\subseteq\CC^2$ isometric.

From \cite[Excercise 1.6.23]{Petersen16} (see also \cite[Section 1.4.5]{Petersen16}  and the discussion in \cite{zhou2024}), we have what follows. We will need only the case $n=4$.
\begin{prop}\label{vefdsca}
	Assume that   $n\in\NN$ and that
	\begin{itemize}
		\item in a in a neighborhood of $0$, $\rho(r)=nr$, and
		\item in a in a neighborhood of $0$, $\phi(r)$ equals to a positive constant. 
			\end{itemize}
			Then, the metric $ds_{\rho,\phi}^2$ is a (smooth) Riemannian metric on (the completion of)
			$(0,\infty)\times (S^3/\mu_n)$.
\end{prop}
The space described in Proposition \ref{vefdsca}  is a vector bundle on $S^2$, in particular, is simply connected (more precisely, there is an obvious deformation retraction from the space to the two sphere given by the Hopf fibration at $r=0$). Moreover, the action of $\mu_n$ on $S^3$ corresponds to the natural $\ZZ/n\ZZ$ action on the Hopf fiber given by the coordinate $\alpha$.
\subsection*{The choice of the metric}
For the following proposition, we refer to  \cite[Section 4.1]{zhou2024}.
\begin{lem}\label{vaefdsc}
	The metric $ds_{\rho,\phi}^2$  has non-negative Ricci curvature if and only, for any $r\in (0,\infty)$, all of the following quantities are non-negative:
	\begin{enumerate}[label=(\arabic*)]
		\item 	$-\frac{\rho''}{\rho}-2\frac{\phi''}{\phi}$,
		\item $2\frac{\rho^4}{\phi^4}-2\frac{\rho}{\phi}\rho'\phi'-\rho\rho''$,
		\item $4-2\frac{\rho^2}{\phi^2}-\frac{\phi}{\rho}\rho'\phi'-\phi\phi''-\phi'\phi'$.
	\end{enumerate}
\end{lem}
Now we choose $\rho$ and $\varphi$. We start by choosing, for $c$ small enough (to be determined later)
\begin{align}
	&\hat\rho =
	\begin{cases}\notag
		r\mapsto nr \qquad &r\in (-1,1),\\
		r\mapsto n+c (r-1)\qquad&r\in (1,\infty);
	\end{cases}\\
	&\hat\phi=
	\begin{cases}\notag
		n\qquad &r\in (-1,1),\\
		r\mapsto n+c (r-1)\qquad&r\in (1,\infty).
	\end{cases}
\end{align}
We then take a radial mollification kernel $\varphi$ supported on $(-1/4,1/4)$ and set, on $[0,\infty)$, $\rho\defeq \varphi\ast\hat\rho$ and similarly $\phi\defeq \varphi\ast\hat\phi$.   
\begin{lem}
	With the choice of $\rho$ and $\phi$ as above, the conditions of Lemma \ref{vaefdsc} are satisfied, provided that $c$ is small enough.
\end{lem}
\begin{proof}
	Notice first that $\hat \rho=\rho$ on $(0,3/4)\cup (5/4,\infty)$ and similarly for $\phi$.
	
	We start by addressing condition $(1)$. Notice first that $\hat \rho''=(c-n)\delta_1$ and $\hat\phi''=c\delta_1$, so that, if $c$ is small enough, $-\hat\rho''-2\hat\phi''\ge 0$. Taking the convolution, we have that $-\rho''-2\phi''\ge 0$, whence $(1)$ follows as $\rho\le \phi$.
	
	For what concerns $(2)$, notice that $\rho$ is concave and therefore
	\begin{equation*}
2\frac{\rho^4}{\phi^4}-2\frac{\rho}{\phi}\rho'\phi'-\rho\rho''\ge 
\begin{cases}
	2\frac{\rho^4}{\phi^4}&\qquad r\in (0,3/4),\\
	2\frac{\rho}{\phi}\Big(\frac{(n/2)^3}{(n+c/2)^3}-nc\Big)&\qquad r\in (3/4,5/4),\\
	2-2c^2&\qquad r\in (5/4,\infty),
\end{cases}
	\end{equation*}
	which is non-negative for $c$ sufficiently small. Finally, for what concerns $(3)$,
	\begin{align*}
			&4-2\frac{\rho^2}{\phi^2}-\frac{\phi}{\rho}\rho'\phi'-\phi\phi''-\phi'\phi'\ge \\&\qquad\qquad
\begin{cases}
	4-2&\qquad r\in (0,3/4),\\
	4-2-\frac{n+c/2}{n/2}nc-(n+c/2)c\|\varphi\|_\infty-c^2&\qquad r\in (3/4,5/4),\\
	4-2-2c^2&\qquad r\in (5/4,\infty),
\end{cases}
	\end{align*}
	which, again,  is non-negative for $c$ sufficiently small.
\end{proof}

\subsection*{Proofs of Theorems \ref{thm2} and  \ref{thm2a}}
	To prove   Theorem \ref{thm2}, we choose  $\rho$ and $\phi$ as described above and we rescale the variable $r$.\qed
	
\noindent
The proof of Theorem \ref{thm2a} is exactly as the proof of Theorem  \ref{thm1a}.\qed

\section{Proof of the  main result}
Fix $c$  smaller than the thresholds given by Theorem \ref{thm1a} and Theorem \ref{thm2a}. Let also $d=d_i=1/i$, for $i\ge 2$. Let $M_i$ be the quotient of the space given by Theorem \ref{thm1a} by the involutive isometry, and let $N_i$ be the space given by Theorem \ref{thm2a}. The conclusion follows, as both $M_i$ and $N_i$ converge, in the GH topology, to the spherical suspension of $cS^3/\mu_4$. 
\qed

\end{document}